\documentclass[11pt]{article}

\usepackage[alphabetic]{amsrefs}
\usepackage{amsmath,amssymb,amsthm}
\usepackage{tikz,enumerate}
\usepackage[all,cmtip]{xy}
\usepackage[applemac]{inputenc}
\usepackage[colorlinks=true, linkcolor=blue, citecolor=blue]{hyperref}

\def \1{\mathds{1}}
\def \a{\mathfrak a}
\def \A{{\mathbb A}}
\def \al{\alpha}
\def \be{\beta}
\def \bs{\backslash}
\def \const{\mathrm{const}}
\def \C{{\mathbb C}}

\def \CB{\mathcal{B}}

\def \CF{\mathcal{F}}

\def \CO{\mathcal{O}}

\def \CS{\mathcal{S}}
\def \CT{\mathcal{T}}

\def \eps{\varepsilon}
\def \F{{\mathbb F}}
\def \fin{\mathrm{fin}}
\def \g{\mathfrak g}
\def \G{{\mathbb G}}
\def \Ga{\Gamma}
\def \GL{\operatorname{GL}}
\def \ga{\gamma}
\def \Hom{\operatorname{Hom}}

\def \la{\lambda}

\def \N{{\mathbb N}}
\def \ol{\overline}

\def \P{{\mathbb P}}
\def \PGL{\operatorname{PGL}}
\def \Q{{\mathbb Q}}
\def \R{{\mathbb R}}
\def \Re{\operatorname{Re}}
\def \SL{\operatorname{SL}}

\def \sm{\smallsetminus}

\def \T{{\mathbb T}}

\def \vol{{\rm vol}}

\def \Z{{\mathbb Z}}
\def \({\left(}
\def \){\right)}

\newcommand{\e}
[1]{\emph{#1}\index{#1}}

\newcommand{\norm}
[1]{\left\|#1\right\|}

\renewcommand{\sp}
[1]{\left\langle #1\right\rangle}

\newcommand{\tto}
[1]{\stackrel{#1}{\longrightarrow}}

\newtheorem{theorem}{Theorem}[section]

\newtheorem{question}[theorem]{Question}
\newtheorem{lemma}[theorem]{Lemma}
\newtheorem{corollary}[theorem]{Corollary}
\newtheorem{proposition}[theorem]{Proposition}

\theoremstyle{definition}
\newtheorem{definition}[theorem]{Definition}
\newtheorem{example}[theorem]{Example}

\begin{document}

\pagestyle{myheadings} \markright{HEIGHT GROWTH}

\title{Height growth on semisimple groups}
\author{Anton Deitmar \& Rupert McCallum}
\date{}
\maketitle

{\bf Abstract:} A  condition is given, under which a general lattice point counting function is asymptotic to the corresponding ball volume growth function.
This is then used to give height asymptotics in the style of the Batyrev-Manin Conjecture for certain intrinsically defined heights on semisimple groups. 

$$ $$

\tableofcontents

\newpage
\section*{Introduction}

Let $V$ be a variety over $\Q$.
A line bundle over $V$ gives rise to a height function $h$ which can be viewed as a means of measuring how ``many'' rational points the variety possesses.
The Batyrev-Manin conjecture \cite{BaMan}, states an asymptotic formula for the number of points of height $\le x$.
It has been shown to hold for flag varieties in \cite{FrManTsch} and for the wonderful compactification of a semisimple linear algebraic group in \cite{Gorodnik}.
A linear algebraic group $\G$ is an affine variety and there is no ``natural'' height function, so one has to construct a height via extra data, such as a projective embedding.

In her doctoral thesis under the supervision of Victor Batyrev \cite{Mennecke}, Susanne Mennecke proposed a more intrinsic definition of a height function and asked for an analogue of the Batyrev-Manin conjecture.
This height function has the advantage of being built from data given by the group $\G$ itself, like its symmetric space and the Bruhat-Tits buildings of the groups of $p$-adic points.
The height function $h$ is defined as the product of local heights $h_v$, which  in turn are defined for any place $v$ by
$$
\log_v h_v(x)=\mathrm{dist}_v(xK_v,K_v),
$$
where $\log_v$ is the natural logarithm if $v$ is archimedean and equals the logarithm to the basis $q_v$ otherwise, where $q_v$ is the cardinality of the residue class field.
The point $K_v$ is the base point in the symmetric space in the archimedean case and the base point of the Bruhat-Tits building otherwise. The distance function is the combinatorial distance in the nonarchimedean case and the Riemannian distance if $v$ is archimedean.
Then the global height $h$ on the adelic points equals
$$
h(x)=\prod_vh_v(x_v).
$$
Let $\pi=\pi_\G$ denote the height counting function,
$$
\pi(x)=\#\Big\{\ga\in \G(\Q): h(\ga)\le x\Big\}.
$$
Mennecke showed that for the group $\G=\SL_2$ one obtains
$$
\pi(x)\sim cx^{3/2}
$$
for some $c>0$. She conjectured a similar asymptotic for $\G=\SL_d$, $d\ge 3$.
In the current paper we approach this problem via a general result on lattice-point problems on metric spaces, relating the lattice point counting to metric ball growth asymptotics.
In order to do computations we modify Mennecke's definition in two ways. First we replace the Riemannian distance at the infinite place by a metric which has canonical Weyl-group invariant polygons as metric balls, a change that seems natural since  it matches up with the combinatorial distance at the finite places.
Next we scale the metric at infinity 
by a free parameter which determines the dominance of either the volume growth at the finite, or the infinite places.
Our result for a general group reads
$$
\pi(x)\ \sim\ \frac {C}{\vol(\G_\Q\bs\G_\A)}\ (B\log x)^{r_\R-1}\ x^B,
$$
where $B$ is the scaling constant, $r_\R$ is the real rank of the group and $C$ is an explicit constant.
This result holds if the scaling constant $B$ is large. For small values of $B$ a general result of this type is harder to come by, since the volume growth $L$-function is not well behaved in this case.
In the case of the rank one groups $\PGL_2$ or $\SL_2$ we can derive an explicit  result confirming Mennecke's computations as follows:
We find that for the rank one group $\G=\PGL_2$ or $\G=\SL_2$ the  volume $b^\fin(T)$ of a ball of radius $T>0$ satisfies
$$
b^\fin(T)\sim c e^{dT},
$$
where, in the case $\G=\SL_2$ we have
$c=\frac{15}{2\pi^2}$ and $d=\frac32$, whereas for $\G=\PGL_2$ we get $c=\frac{15}{\pi^2}$ and $d=2$.  
If we then scale the metric at infinity with a scaling parameter $0<B<d$, then we get, as $x\to\infty$,
$$
\pi(x)\ \sim\ cd\int_0^\infty e^{-dt}b^\infty(t)\,dt\frac1{\vol(\G_\Q\bs\G_\A)}x^d.
$$
As for the wonderful compactification the local heights are at almost all places given by formula similar to ours (see Lemma 6.4 in \cite{STT}), these results are not too far from what has been done in the papers \cites{Gorodnik, STT, TT}.

\section{Lattice points and ball volume}

Recall the notion of a \e{matrix coefficient} of a unitary representation $(\pi,V_\pi)$ of a locally compact group $G$ on some Hilbert space $V_\pi$:
this is any function of the form $\psi_{v,w}:G\to\C$, $x\mapsto \sp{\pi(x)v,w}$ for some $v,w\in V_\pi$.

For a compact subgroup $K\subset G$ we write $V_\pi^K$ for the set of $K$-invariant vectors in $V_\pi$. 
We say that a matrix coefficient $\psi$ is $K$-invariant, if $\psi(kx)=\psi(xk)=\psi(x)$ holds for every $x\in G$ and all $k\in K$.
If this is the case, then $\psi=\psi_{v,w}$ for some $v,w\in V_\pi^K$.
To see this, assume that $\psi=\psi_{\tilde v,\tilde w}$ is $K$-invariant, then with the normalised Haar measure on $K$ we get $\psi=\psi_{v,w}$ with $v=\int_K\pi(k)\tilde v\,dk$ and $w=\int_K\pi(k)\tilde w\,dk$.

\begin{definition}
Let $G$ be a locally compact group with a fixed Haar measure $dx$.
Let $\Ga\subset G$ be a lattice, i.e., $\Ga$ is a discrete subgroup such that on $\Ga\bs G$ there exists a $G$-invariant Radon measure of finite volume.
Let $L_0^2(\Ga\bs G)$ denote the orthogonal space in $L^2(\Ga\bs G)$ of the one-dimensional space of constant functions, i.e.,
$$
L_0^2(\Ga\bs G)=\left\{ f\in L^2(\Ga\bs G): \int_{\Ga\bs G}f(x)\,dx=0\right\}.
$$
Let $K\subset G$ be a compact subgroup.
We say that $G$ acts \e{$K$-mixingly} on $\Ga\bs G$, if every $K$-invariant matrix coefficient in the Hilbert space $L_0^2(\Ga\bs G)$ vanishes at infinity.
\end{definition}

\begin{definition}
Let $b:[0,\infty)\to [0,\infty)$ be monotonically increasing with $\lim_{T\to\infty}b(T)=\infty$.
We say that $b$ is \e{regular}, if
$$
\lim_{\eps\searrow 0}\liminf_{T\to\infty}\frac{b(T-\eps)}{b(T)}=\lim_{\eps\searrow 0}\limsup_{T\to\infty}\frac{b(T+\eps)}{b(T)}.
$$
If this happens, then both limits are equal to $1$.
Note that regularity only depends on the asymptotic class, i.e., if $b(x)\sim b_1(x)$  as $x\to\infty$, then $b_1$ is regular if and only if $b$ is regular.
\end{definition}

\begin{example}
For any given $a,b\ge 0$, not both zero, the function $b(x)=x^ae^{bx}$ is regular.
The function $b(x)=e^{x^2}$ is not.
Also, the function $b(x)=e^{[x]}$ is not regular, where $[x]$ is the largest integer $\le x$.
\end{example}

\begin{theorem}\label{prop1.5}
Let $G$ be a locally compact group with a lattice $\Ga\subset G$.
Let $K\subset G$ be a compact group and assume that on $X=G/K$ there is given a $G$-invariant proper metric $d$.
Let $b(T)$ denote the Haar measure of the closed ball $\ol B_T$ of radius $T>0$ around $eK$ in $X$.
Let 
$$
N(T)=\#\Big\{ \ga\in\Ga: d(\ga K,e K)\le T\Big\}.
$$
Assume that
\begin{enumerate}[\rm (a)]
\item $G$ acts $K$-mixingly on $\Ga\bs G$ and
\item $b(T)$ is regular.
\end{enumerate}
Then we have, as $T\to\infty$,
$$
N(T)\ \sim\ \frac1{\vol(\Ga\bs G)}\, b(T).
$$
\end{theorem}

It may happen that the function $b$ is not regular per se, but there exists an unbounded subset $\CS\subset (0,\infty)$, such that $b$ is regular on $\CS$, i.e., one has
$$
\lim_{\eps\searrow 0}\liminf_{\substack{T\to\infty\\ T\in \CS}}\frac{b(T-\eps)}{b(T)}=\lim_{\eps\searrow 0}\limsup_{\substack{T\to\infty\\ T\in \CS}}\frac{b(T+\eps)}{b(T)}.
$$
An example is the function $b(x)=e^{[x]}$.
In that case, the asymptotic assertion remains true when restricted to $\CS$, i.e., one then has $\lim_{\substack{T\to\infty\\ T\in\CS}}\frac{N(T)}{b(T)}\vol(\Ga\bs G)=1$, which we denote as
$$
N(T)\ \sim_\CS\ \frac1{\vol(\Ga\bs G)}\, b(T).
$$

\begin{proof}
Fix some smooth, monotonically decreasing function $\chi:\R\to \R$ such that 
\begin{itemize}
\item $0\le\chi\le 1$,
\item $\chi(x)=1$, if $x<-1$,
\item $\chi(x)=0$, if $x>1$.
\end{itemize}
For $T>1$ let $\chi_T(x)=\chi(e^{T}(x-T))$.
Then we in particular have
$$
\chi_T(x)=\begin{cases} 1&x\le T-e^{-T},\\
0&x\ge T+e^{-T}.\end{cases}
$$

\begin{lemma}\label{lem1.5}
For every $\eps>0$ there is $T_0$ such that for $T\ge T_0$ we have
$$
\chi_{T-\eps}(x)\le \chi_T(x+\eps)\quad\text{and}\quad \chi_T(x-\eps)\le \chi_{T+\eps}(x).
$$
\end{lemma}

\begin{proof}
The second claim follows from the first, so we only show the first.
This claim is equivalent to
$$
\chi(e^{T-\eps}(x-T+\eps))\le\chi(e^T(x-T-\eps)).
$$
If $e^{T-\eps}(x-T+\eps)\ge 1$, then the left hand side is zero and the claim follows.
Otherwise, we have $x\le e^{\eps-T}+T-\eps$ and so
\begin{align*}
(e^\eps-1)x-e^\eps\eps &\le
(e^\eps-1)(e^{\eps-T}+T-\eps)-e^\eps\eps\\
&=(e^\eps-1)T+\underbrace{\((e^\eps-1)(e^{\eps-T}-\eps)-e^\eps\eps\)}_{<0\text{ for }T\ll 0}.
\end{align*}
The second summand has a limit $<0$ as $T\to\infty$, therefore there exists $T_0$ such that this summand is $<\eps$ for $T\ge T_0$.
Therefore, for $T\ge T_0$ we have
$$
(e^\eps-1)x-e^\eps\eps
\le (e^\eps-1)T+\eps,
$$
or
$$
e^T(x-T-\eps)\le e^{T-\eps}(x-T+\eps)
$$
As $\chi$ is monotonically decreasing, the claim follows.
\end{proof}

Define a function $f_T:G\to\C$ by $f_T(x)=\chi_T(d(xK,K))$ and let $k_T(x,y)=\sum_{\ga\in \Ga}f_T(x^{-1}\ga y)$.
Then for any $\phi\in L^2(\Ga\bs G)$ we use the quotient integral formula 
\cite{HA2}*{Chapter 1} to compute
\begin{align*}
R(f_T)\phi(x)&=\int_Gf_T(y)\phi(xy)\,dy\\
&=\int_Gf_T(x^{-1}y)\phi(y)\,dy\\
&=\int_{\Ga\bs G}\sum_{\ga\in\Ga}f_T(x^{-1}\ga y)\phi(y)\,dy\\
&=\int_{\Ga\bs G}k_T(x,y)\phi(y)\,dy.
\end{align*}
So $k_T\in C(\Ga\bs G\times\Ga\bs G)$ is the integral kernel of the operator $R(f_T)$.

Let  $\la_T$ denote the eigenvalue of $R(f_T)$ on the space of constant functions. The latter has the function $\phi_\const\equiv\frac1{\sqrt{\vol(\Ga\bs G)}}$ for an orthonormal basis. We have
\begin{align*}
\la_T&=\sqrt{\vol(\Ga\bs G)}\la_{T}\phi_\const(1)\\
&=\sqrt{\vol(\Ga\bs G)} R(f_T)\phi_\const(1)\\
&=\sqrt{\vol(\Ga\bs G)}\int_{G}f_T(x)\phi_\const(x)\,dx\\
&=\int_{G}f_T(x)\,dx.
\end{align*}
The estimate $b(T)\le \int_G f_T\le b(T+e^{-T})$ leads to
$$
1\le \frac{\int_G f_T}{b(T)}\le \frac{b(T+e^{-T})}{b(T)}\le \frac{b(T+\eps)}{b(T)}
$$
for every given $\eps>0$, if $T>>0$.
As $b$ is regular, we infer
$$
\la_T\sim b(T),\quad\text{as}\ T\to\infty.
$$

Let now $U$ be an open neighborhood of the unit in $G$ such that $\ga U\cap U=\emptyset$ for every $\ga\in \Ga\sm\{1\}$.
Let $\eta\in C_c(\Ga\bs G)$ be 
of support inside $\Ga U$, such that $\eta\ge 0$ and $\int_{\Ga\bs G}\eta(x)\,dx=1$.
Write
$$
\eta=\eta_\const\oplus\eta_0,
$$
where $\eta_\const$ stands for the projection onto the space of constant functions, i.e., $\eta_\const=\sp{\eta,\phi_\const}\phi_\const=\frac1{{\vol(\Ga\bs G)}}\int_{\Ga\bs G}\eta(x)\,dx=\frac1{{\vol(\Ga\bs G)}}$.
Further, $\eta_0\in L^2_0(\Ga\bs G)$.
Then 
$$
R(f_T)\eta=\frac{\la_T}{{\vol(\Ga\bs G)}}\oplus R(f_T)\eta_0
$$
and so
$$
\sp{R(f_T)\eta,\eta}=\frac{\la_T}{{\vol(\Ga\bs G)}}+\sp{R(f_T)\eta_0,\eta_0}
$$
On the other hand, let $\CF\subset G$ be a measurable set of representatives for $\Ga\bs G$, where we can assume that $U\subset \CF$. Then integrating a $\Ga$-invariant function over $\Ga\bs G$ means the same as integrating it over $\CF$ and so
\begin{align*}
\sp{R(f_T)\eta,\eta}&=\int_{\Ga\bs G}\int_{\Ga\bs G}k_T(x,y)\eta(y)\,dy\,\eta(x)\,dx\\
&=\int_{\CF}\int_{\CF}k_T(x,y)\eta(y)\,dy\,\eta(x)\,dx\\
&=\int_{U}\int_{U}k_T(x,y)\eta(y)\,dy\,\eta(x)\,dx\\
&=\sum_{\ga\in \Ga}\int_U\int_U f_T(x^{-1}\ga y)\,\eta(x)\eta(y)\,dx\,dy.
\end{align*}
Now 
\begin{align*}
f_T(x^{-1}\ga y)&=\chi_T(d(x^{-1}\ga yK,K))\\
&= \chi_T(d(\ga yK,xK)),
\end{align*}
and we have
\begin{align*}
|d(\ga yK,xK)-d(\ga K,K)|&\le |d(\ga yK,xK)-d(\ga yK,K)|\\
&\ \ \ + |d(\ga yK,K)-d(\ga K,K)|\\
&\le d(xK,K)+d(\ga yK,\ga K)\\
&=d(xK,K)+d(yK,K)
\end{align*}

Let $\eps>0$ and suppose that $U$ is so small that for every $x\in U$ we have $d(xK,K)<\eps/2$.
Then for any $x,y\in U$ we have
$$
d(\ga xK,yK)-\eps\le d(\ga K,K)\le d(\ga xK,yK)+\eps.
$$
As $\chi_T$ is monotonically decreasing, it follows
$$
\chi_T(d(\ga xK,yK)+\eps)
\le \chi_T(d(\ga K,K))
\le \chi_T(d(\ga xK,yK)-\eps).
$$
By Lemma \ref{lem1.5} we get for $T$ large enough,
$$
\chi_{T-\eps}(d(\ga xK,yK))
\le \chi_T(d(\ga K,K))
\le \chi_{T+\eps}(d(\ga xK,yK)).
$$
Integrating and summing we get for $T>>0$,
$$
\sp{R(f_{T-\eps})\eta,\eta}\le N(T)\le \sp{R(f_{T+\eps})\eta,\eta}.
$$
In order to show the proposition, we need to estimate the term
$$
\sp{R(f_{T\pm\eps})\eta_0,\eta_0}=\sp{R(f_{T\pm\eps})P\eta_0,P\eta_0},
$$ 
where $P$ is the orthogonal projection to the set of $K$-invariants.
The $K$-mixing property implies that the matrix coeffient $x\mapsto \sp{R(x)P\eta_0,P\eta_0}$ vanishes at infinity.
So let $\delta>0$, then there exists a compact set $C\subset G$ such that 
$|\sp{R(x)P\eta_0,P\eta_0}|<\delta$ for $x\notin C$.
Let $G_T$ denote the set of all $g\in G$ such that $f_T(x)=1$. 
As $T\to\infty$, $G_T$ absorbs any compact set, so we may choose $T_0$ so large that $C\subset G_{T_0-\eps}$.
For $T\ge T_0$ we have
\begin{align*}
\frac1{b(T)}|\sp{R(f_{T\pm\eps})\eta_0,\eta_0}|
&=\frac1{b(T)}\left|\int_{G}f_{T\pm\eps}(x)\sp{R(x)P\eta_0,P\eta_0}\,dx\right|\\
&\le\frac1{b(T)}\int_{G}f_{T\pm\eps}(x)|\sp{R(x)P\eta_0,P\eta_0}|\,dx\\
&\le \frac1{b(T)}\(b(T_0)+\delta b(T+\eps)\)\\
&=\frac{b(T_0)(1-\delta)}{b(T)}+\delta\frac{b(T+\eps)}{b(T)}.
\end{align*}
The first summand tends to zero as $T\to\infty$.
From this we conclude
$$
\frac{N(T)}{b(T)}\le \frac{|\la_T|}{b(T)}+\frac{b(T_0)(1-\delta)}{b(T)}+\delta\frac{b(T+\eps)}{b(T)},
$$
which implies that
$$
\limsup_{T\to\infty}\frac{N(T)}{b(T)}=1.
$$
The limes inferior is dealt with analogously.
\end{proof}

\begin{corollary}\label{Cor1.5}
Let the situation be as in the theorem, except that $b(T)$ is not necessarily regular.
Then we conclude that for every $\eps>0$ we have, as $T\to\infty$,
$$
b(T-\eps)\ll N(T)\ll b(T+\eps).
$$
\end{corollary}

\begin{proof}
This follows from the proof of the theorem.
\end{proof}

\begin{definition}
By a \e{homogeneous metric measure space} we mean a quadruple $(G,X,d,\mu)$ consisting of a group $G$, a non-compact proper metric space $(X,d)$ on which $G$ acts transitively and isometrically, and a $G$-invariant Radon measure $\mu\ne 0$ on $X$.
\end{definition}

\begin{definition}
The metric $d$ is called an \e{inner metric} or \e{length metric}, if the distance $d(x,y)$ equals the infimum of lengthes of rectifiable curves joining $x$ and $y$.
\end{definition}

\begin{question}
Let $(G,X,d,\mu)$ be a homogeneous metric measure space and assume that the metric is inner.
Let $x_0\in X$ and let
$$
b(T)=\mu\(\Big\{x\in X: d(x,x_0)\le T\Big\}\).
$$
Is it true that the function $b$ is regular?
\end{question}

\begin{example}
Let $\CT$ be a regular tree of valency $q+1$, $q\in\N$, $q\ge 2$ and let $X$ be the set of vertices of $\CT$, equipped with the path length metric of $\CT$. The Automorphism group $G$ of $\CT$ acts isometrically on $X$ and with $\mu$ being the counting measure, $(G,X,d,\mu)$ is a homogeneous metric measure space.
In this case we have for $T\ge 1$,
$$
b(T)=1+(q+1)q^{[T]},
$$
where $[T]$ is the largest integer $k\le T$.
For $n\in\N$, $n\ge 3$ and $0<\eps<1$ we have
\begin{align*}
\frac{b(n-\eps)}{b(n)}
&=\frac{1+(q+1)q^{n-2}}{1+(q+1)q^{n-1}}=\frac1q\frac{q+1+\frac1{q^{n-2}}}{q+1+\frac1{q^{n-1}}}\ \tto{n\to\infty}\ \frac1q.
\end{align*}
Therefore, the function $b(T)$ is not regular.
\end{example}

\section{Small overlaps}
Let $(G,X,d,\mu)$ be a homogeneous metric measure space.
Write
$$
\ol B_T(x)=\{ y\in X:d(x,y)\le T\}
$$
be the closed $T$-ball around $x$.
For $T,\delta>0$ let $[\ol B_T:\delta]$ be defined as
$$
[\ol B_T:\delta]=\min\left\{ k\in\N: \ol B_T\subset \bigcup_{j=1}^k \ol B_\delta(x_j),\ x_1,\dots,x_k\in X\right\}.
$$
For any $T>0$ write 
$$
b(T)=\mu\(\ol B_T(x)\)
$$
for any $x\in X$ (doesn't depend on $x$).

\begin{definition}
We say that $X$ has \e{small overlaps}, if  the metric $d$ is inner and there exists a $T_0>0$ such that 
$$
\frac{[\ol B_T:\delta]b(\delta)}{b(T)}\ \tto{\delta\to 0}\ 1
$$
uniformly for $T\in [T_0,\infty)$.
\end{definition}

\begin{example}
The euclidean metric on $\R^n$, $n\ge 2$ has no small overlaps.
The metric on $\R^n$ given by the 1-norm
$$
\norm x_1=\sum_{j=1}^n |x_j|\quad \text{or}\quad \norm x_\infty =\max_j|x_j|
$$ both have small overlaps.
\end{example}
 
\begin{example}\label{ExRootnorm}
Let $\a$ denote a finite-dimensional euclidean vector space and let $\phi\subset\a^*=\Hom(\a,\R)$ be an irreducible root system.
For $\al\in\Phi$ and $k\in\N$ define the affine hyperplane
$$
H_{\al,k}=\{ x\in\a: \al(x)=k\}.
$$
The reflections at all $H_{\al,k}$ generate a group of affine isometries, the \e{Weyl group} $W$.
The set $\a\sm\bigcup_{\al,k}H_{\al,k}$ is a union of open simplices, called \e{affine chambers}.
Let $\ol B_1$ denote the closure of the union of those chambers $C$ which have zero in their closure.
Then $\ol B_1$ is a closed convex neighborhood of zero, so there exists exactly one norm $\norm .$ on $\a$ such that $\ol B_1=\{ v\in\a: \norm v\le 1\}$.
The metric on $\a$, given by this norm has small overlaps.
\end{example}

\begin{proposition}\label{lem2.4}
If $X$ is inner and  has small overlaps, then the function $b(T)$ is regular.
\end{proposition}

\begin{proof}
Let $\delta,\eps,\tau>0$.
Cover $\ol B_T(x_0)$ with balls of radius $\delta$ around the midpoints $x_1,\dots,x_k$.
As the metric is inner, the balls of radius $\delta+\eps+\tau$ around $x_1,\dots,x_k$  cover $\ol B_{T+\eps}(x_0)$. 
So we have
\begin{align*}
\frac{b(T+\eps)}{b(T)} &\le \frac{[\ol B_T:\delta]b(\delta+\eps+\tau)}{b(T)}.
\end{align*}
Let $\alpha>0$ be given and choose $\delta>0$ so small that for every $T\ge T_0$ one has 
$$
\left|\frac{[\ol B_T:\delta]b(\delta)}{b(T)}- 1\right|<\alpha.
$$
Then for $T\ge T_0$ we get
\begin{align*}
\frac{b(T+\eps)}{b(T)} 
&\le(1+\al)\frac{b(\delta+\eps+\tau)}{b(\delta)},
\end{align*}
so that
$$
1\le \lim_{\eps\to 0}\limsup\frac{b(T+\eps)}{b(T)}\le (1+\al)\frac{b(\delta+\tau)}{b(\delta)},
$$
and as $\al$ and $\tau$ are arbitrary, we get that the limit equals 1.
Analogously one gets $\lim_{\eps\to 0}\liminf_{T\to\infty}\frac{b(T-\eps)}{b(T)}=1$.
\end{proof}

The following lemma is of importance to our calculations.

\begin{lemma}[Persistence of the dominant asymptotic]\label{lemMeasure}
Let $\mu,\nu$ be locally finite measures on $[0,\infty)$ such that, for some $\al\ge 0$ and some $\beta>0$, as $T\to\infty$, one has
$$
\nu([0,T])\ \sim\ T^\al e^{\be T}
\quad
\text{and}\quad 
C=\int_{0,\infty)}e^{-\be x}\,d\mu(x)<\infty.
$$
Let
$$
D(T)=\{(x,y)\in\R^2: 0\le x,y\le x+y\le T\}.
$$
Then, as $T\to\infty$, we have
$$
\int_{D(T)}1\ d(\mu\otimes\nu)\ \sim\ C\ T^\al e^{\be T}.
$$
\end{lemma}

\begin{proof}
First note that the assumptions imply that $\mu([0,T])e^{-\be T}$ tends to zero as 
$T\to\infty$.
This is seen as follows: Let $0<S<T$, then
\begin{align*}
\mu([0,T])e^{-\be T}
&=\int_{[0,T]} \frac{e^{-\be t}}{e^{\be{T-t}}}\,d\mu(t)\\
&\le \frac1{e^{\be(T-S)}}\int_{[0,S]}e^{-\be t}\,d\mu(t)+\int_{(S,T]}e^{-\be t}\,d\mu(t).
\end{align*}
This implies that
$$
\limsup_{T\to\infty}
\mu([0,T])e^{-\be T}\le \int_{(S,\infty)}e^{-\be t}\,d\mu(t).
$$
As the left had side does not depend on $S$, it is zero.
Next write $d(T)=\int_{D(T)}d(\mu\otimes\nu)$.
Let $T_0>0$ be so large that for every $T\ge T_0$ one has
$$
\left|\frac{\nu([0,T])}{T^\al e^{\be T}}-1\right|<\eps.
$$
For $T>2T_0$ we have
\begin{align*}
\frac{d(T)}{T^\al e^{\be T}}
&=\int_{[0,T]}\frac{\nu([0,T-t])}{T^\al e^{\be T}}\,d\mu(t)\\
&=\int_{[0,T_0]}\frac{\nu([0,T-t])}{T^\al e^{\be T}}\,d\mu(t)
+\int_{(T_0,T]}\frac{\nu([0,T-t])}{T^\al e^{\be T}}\,d\mu(t)\\
&=\underbrace{\int_{[0,T_0]}\frac{\nu([0,T-t])}{T^\al e^{\be T}}\,d\mu(t)}_{=a(T,T_0)}
+\underbrace{\int_{(T_0,T]}\frac{\nu([0,T_0])}{T^\al e^{\be T}}\,d\mu(t)}_{=b(T,T_0)}\\
&\quad +\underbrace{\int_{(T_0,T]}\frac{\nu((T_0,T-t])}{T^\al e^{\be T}}\,d\mu(t)}_{=c(T,T_0)}.
\end{align*}
These three summands are obtained by integrating $\mu\otimes\nu$ over the areas $A,B$ and $C$ as in the following picture, which is why we call them $a,b,c$.
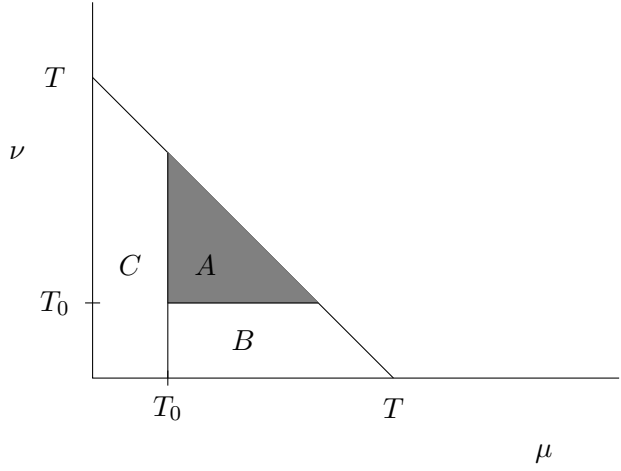
\begin{center}
\begin{tikzpicture}
\draw(0,5)--(0,0)--(7,0);
\draw(0,4)--(4,0);
\filldraw[gray](3,1)--(1,1)--(1,3)--(3,1);
\draw(1,-.1)--(1,.1);
\draw(1,-.4)node{$T_0$};
\draw(4,-.4)node{$T$};
\draw(1.5,1.5)node{$A$};
\draw(1,0)--(1,3);
\draw(.5,1.5)node{$C$};
\draw(2,.5)node{$B$};
\draw(6,-1)node{$\mu$};
\draw(-1,3)node{$\nu$};
\draw(-.1,1)--(.1,1);
\draw(-.5,1)node{$T_0$};
\draw(-.5,4)node{$T$};
\draw(1,1)--(3,1);
\end{tikzpicture}
\end{center}
The summand $b(T,T_0)$ tends to zero as $T\to\infty$ by the first part of the proof.
The first summand is
\begin{align*}
a(T,T_0)&=\int_{[0,T_0]}\frac{\nu([0,T-t])}{(T-t)^\al e^{\be (T-t)}}\frac{(T-t)^\al e^{\be (T-t)}}{T^\al e^{\be T}}\,d\mu(t)\\
&\le (1+\eps)\int_{[0,T_0]}\frac{(T-t)^\al e^{\be (T-t)}}{T^\al e^{\be T}}\,d\mu(t)\\
&=(1+\eps)\int_{[0,T_0]}e^{-\be t}\frac{(T-t)^\al}{T^\al}\,d\mu(t)\\
&\le (1+\eps)\int_{[0,T_0]}e^{-\be t}\,d\mu(t),
\end{align*}
so that
$$
\limsup_{T\to\infty}a(T,T_0)\le (1+\eps)\int_{[0,T_0]}e^{-\be t}\,d\mu(t).
$$
On the other hand we have that
\begin{align*}
a(T,T_0)&\ge (1-\eps) \int_{[0,T_0]}e^{-\be t}\frac{(T-t)^\al}{T^\al}\,d\mu(t)\\
&\ge (1-\eps) \int_{[0,T_0]}e^{-\be t}\,d\mu(t)\frac{(T-T_0)^\al}{T^\al}
\end{align*}
so that
$$
\liminf_{T\to\infty}a(T,T_0)\ge (1-\eps)\int_{[0,T_0]}e^{-\be t}\,d\mu(t).
$$
As for the third summand we have
\begin{align*}
c(T,T_0)
&=\int_{(T_0,T-T_0]}\frac{\nu((T_0,T-t])}{T^\al e^{\be T}}\,d\mu(t)\\
&\le \int_{(T_0,T-T_0]}\frac{\nu([0,T-t])}{T^\al e^{\be T}}\,d\mu(t)\\
&= \int_{(T_0,T-T_0]}e^{-\be t}\frac{\nu([0,T-t])}{(T-t)^\al e^{\be (T-t)}}
\frac{(T-t)^\al}{T^\al}\,d\mu(t)\\
&\le (1+\eps)
\int_{(T_0,T-T_0]}e^{-\be t}
\frac{(T-t)^\al}{T^\al}\,d\mu(t)\\
&\le (1+\eps)\int_{[T_0,\infty)}e^{-\be t}\,d\mu(t).
\end{align*}
Summing up, we get
\begin{align*}
(1-\eps)\int_{[0,T_0]}e^{-\be t}\,d\mu(t)&\le\liminf_{T\to\infty}\frac{d(T)}{T^\al e^{\be T}}\\
&\le \limsup_{T\to\infty}\frac{d(T)}{T^\al e^{\be T}}
\le (1+\eps)\int_{[0,\infty)}e^{-\be t}\,d\mu(t)
\end{align*}
Letting $\eps$ tend to zero and $T_0$ to infinity accordingly, we get
$$
\lim_{T\to\infty}\frac{d(T)}{T^\al e^{\be T}}=\int_{[0,\infty)}e^{-\be T}\,d\mu(t).
$$
The lemma follows.
\end{proof}

\begin{corollary}\label{Cor2.6}
Let the setting be as in the lemma, except that we only have 
$$
\nu([0,T])\ll T^\al e^{\be T}
$$
as $T\to\infty$.
Then we get
$$
\int_{D(T)}1\ d(\mu\otimes\nu)\ \ll\ C\ T^\al e^{\be T}
$$
and the same for $\gg$.
\end{corollary}

\begin{proof}
Inspection of the proof of the lemma.
\end{proof}

\section{The adelic metric space}
\begin{definition}
Let $\G$ denote a semisimple linear algebraic group over $\Q$.
For a ring extension $R/\Q$ we write $\G_R$ for the group of $R$-valued points.
By $p$ we denote either a prime number, in which case we write $p<\infty$ or $p=\infty$.
In both cases we say that $p$ is a \e{place} of $\Q$.
We write $\Q_p$ for the completion of $\Q$ at $p$, so for instance $\Q_\infty=\R$.
\end{definition}

\begin{definition}
Let $\A=\A_\fin\times\A_\infty$ denote the adele ring of $\Q$, where $\A_\infty=\R$ and $\A_\fin$ is the ring of finite adeles, i.e., the restricted product of all $\Q_p$ with $p<\infty$.
Writing $\G_\fin=\G_{\A_\fin}$ and $\G_\infty=\G_{\R}$, we have $G=\G_\A=\G_\fin\times \G_\infty$.
We abbreviate  $\G_{\Q_p}$ by $\G_p$.
Fix an embedding $\G\hookrightarrow \GL_n$ over $\Q$.
For every $p<\infty$ set $\G_{\Z_p}= \G_p\cap\GL_n(\Z_p)$.
For each $p\le\infty$ we fix a maximal compact subgroup $K_p$ of $\G_p$ in such a way that 
\begin{itemize}
\item if $p<\infty$, then $K_p$ is special, i.e., it fixes a special vertex in the Bruhat-Tits building,
\item for almost all $p$, the group $K_p$ is hyperspecial and we have $K_p=\G_{\Z_p}$.
\end{itemize}
For the possibility of such a choice, see \cite{TitsCorv}*{3.9.1}.
We need to fix Haar measures.
For every $p\le\infty$ we normalize the Haar measure on $K_p$ by insisting that 
$\vol(K_p)=1$.
For $p<\infty$ this already fixes the Haar measure on $G_p$.
On  $\G_\infty$ any Haar measure is induced from a $\G_\infty$-invariant Riemannian metric on the symmetric space $X_\infty=\G_\infty/K_\infty$, which we choose to be the one given by the Killing form $b$.
\end{definition}

\begin{definition}
For every $p<\infty$ we denote by $\CB_p$ the Bruhat-Tits building of  the group $\G_p$.
The building $\CB_p$ carries a metric which is euclidean on each apartment.
For arithmetic purposes it has turned out that, as far as only vertices of the building are concerned, the combinatorial distance is more useful, see \cites{Manin,Werner}.
This is a metric on the set $\CB_{p,0}$ of all vertices of the building defined by considering the graph $\CB_{p,1}$, which is the 1-skeleton of $\CB_p$.
For two vertices $x,y$ of $\CB_p$ let $d_p(x,y)$ denote the combinatorial distance in $\CB_{p,1}$, i.e., the length of the shortest path joining $p$ and $q$, where each edge gets the length 1.
Note that the group $K_p$ fixes a unique vertex $o_p$ in $\CB_p$ and in this way we get a canonical injection $\G_p/K_p\hookrightarrow \CB_{p,0}$, $xK_p\mapsto xo_p$.
We will denote the point $xo_p\in\CB_p$ also by $xK_p$.
\end{definition}

\begin{definition}\label{Defmetrics}
Also for $p=\infty$, there is a metric on the symmetric space $X_\infty=\G_\infty/K_\infty$. However, as at the finite places, it turns out, that for our purposes  a different metric is more useful. 
First observe that for given choice of a maximal compact subgroup $K_\infty$, there exists a maximal connected $\R$-split torus $A$ in $G$ such that $\theta(a)=a^{-1}$ for every $a\in A$, where $\theta$ is the Cartan involution on $G$, fixing $K_\infty$ pointwise.
Let $\a$ be the Lie algebra of $A$ and let $\Phi\subset\a^*$ be the set of roots of $(\a,\g)$, where $\g$ is the Lie algebra of $\G_\infty$. 
For each $\al\in\Phi$ let $m_\al$ denote the dimension of the coresponding root space $\g_\al$.
Fix a choice of positve roots $\Phi^+\subset\Phi$ and let $\rho=\frac12\sum_{\al\in\Phi^+}m_\al\al$ be the modular shift.
Choose some $B>0$ and set
$$
\norm X_{B}=\frac1{2B}\sum_{\al\in\Phi^+}m_\al|\al(X)|.
$$
Then $\norm._{B}$ is a Weyl group invariant norm on $\a$.
For $X\in\a^+$, the positive Weyl chamber given by the ordering, we get $\norm X_{B}=\frac{\rho(X)}{B}$.

We define a $\G_\infty$-invariant metric $d_\infty$ on $X_\infty$ by setting 
$$
d_\infty(xK_\infty,yK_\infty)=\tilde d_{\infty}(y^{-1}xK),
$$
where the one argument function $\tilde d_\infty$ is defined by
$$
\tilde d_\infty(K\exp(X)K)=\norm{X}_{B},\qquad X\in\a.
$$
\end{definition}

\begin{definition}
On the group $\G_p$, $p\le\infty$, we now define the local height $h_p:\G_p\to [1,\infty)$ by the formula
$$
h_p(x)=\begin{cases}p^{d_p(xK_p,K_p)}&p<\infty,\\ e^{d_\infty(xK_\infty,K_\infty)}& p=\infty.\end{cases}
$$
The global height on $\G_\A$ is defined by
$$
h(x)=\prod_{p\le\infty}h_p(x).
$$
We denote the corresponding counting function by $\pi$. So
$$
\pi(x)=\#\{\ga\in\G_\Q:h(\ga)\le x\}.
$$
Let $X_\infty$ be the symmetric space $\G_\infty/K_\infty$ and for $p<\infty$ set $X_p=\CB_{p,1}$.
Let
$$
X=\bigoplus_{p\le\infty}X_p
$$
denote the set of all $x\in \prod_{p\le\infty}X_p$ such that $x_p=K_p$ holds for almost all $p$.
Let $x\in X$ then the stabilizer $K_x$ of $x$ in $\G_\A$ is a maximal compact subgroup of $\G_\A$ which stabilizes no other point, hence the map $x\mapsto K_x$ is an injection of $X$ into the set of all maximal compact subgroups of $\G_\A$. We occasionally denote elements of $X$ as compact subgroups of $\G_\A$.
In particular, the group $K_\A=\prod_{p\le\infty}K_p$ defines a base-point in $X$.
We can identify $X$ with $\G_\A/K_\A$.
\end{definition}

\begin{proposition}\label{lem3.1}
For $p=\infty$ we formally write $\log p=\log e=1$.
The sum
$$
d(x,y)=\sum_{p\le\infty}d_p(x_p,y_p)\log p
$$
is finite for each given pair $x, y\in X$ and defines a metric on $X$ which makes $X$ a complete and proper metric space.
The natural action of $\G_\A$ in $X$ is isometric and the metric topology equals the topology induced by $X\cong \G_\A/K_\A$.
Further, the $\G_\infty$-invariant measure on $\G_\infty/K_\infty$ and the counting measures on the $X_p$ for $p<\infty$ give a $G$-invariant, non-trivial Radon measure on $X$, which coinsides with the quotient of the Haar measures on $\G_\A$ and $K_\A$.
\end{proposition}

\begin{proof}
To show completeness, let $(x_i)$ be a Cauchy sequence and let $i_0\in\N$ be such that for all $i,j\ge i_0$ we have $d(x_i,x_j) < \log 2$. Then for any given $p<\infty$ we have
$$
d_p(x_{i,p},x_{j,p})\log p\le \sum_{q\le\infty} d_q(x_{i,q},x_{j,q})\log q = d(x_i,x_j) < \log 2
$$
So $d_p(x_{i,p}, x_{j,p}) < \frac{\log 2}{\log p} \le 1$ and therefore $d_p(x_{i,p}, x_{j,p}) = 0$, which amounts to $x_{i,p} = x_{j,p}$ for all $i,j\ge i_0$.
Since this holds for every prime $p<\infty$, the completeness follows from the completeness of the metric space $X_\infty$.

For properness, it suffices to show that every closed Ball $B_R(K)$ for given $R>1$ is compact.
For this let $P$ denote the set of all places $p$ such that $\log p\le R$. Then the set $P$ is finite.
Let $x\in B_R(K)$, then for each $p\notin P$ we have
$$
d_p(x_pK_p,K_p)\log p\le d(xK,K)\le R.
$$
Therefore $d_p(x_pK_p,K_p)<1$ hence it is zero. So the ball $B_R(K)$ is homoemorphic to the corresponding ball in $\prod_{v\in P}X_p$ which  is compact.
The rest of the lemma is clear.
\end{proof}

\section{Lattice points of $\Q$-groups}\label{Sec2}

We make the following assumptions.

\begin{enumerate}[\rm (a)]
\item We assume that $\G$ is connected and  almost $\Q$-simple.
\item We assume $\G_\R$ is not compact and the group homomorphism $\tilde\G_\A\to\G_\A$ is surjective, where $\tilde\G$ is the simply connected covering group of $\G$.
\end{enumerate}

The second part of condition (b) is satisfied if $\G$ is itself simply connected, but also in some other cases like $\G=\PGL_n$, see Section \ref{PGLn}.

\begin{definition}
By an \e{automorphic character} of $\G$ we mean a continuous group homomorphism $\chi:G=\G_\A\to \T=\{z\in\C:|z|=1\}$ such that $\chi(\Ga)=1$, where $\Ga=\G_\Q$.
Since $\vol(\G_\Q\bs\G_\A)<\infty$, it follows that $\chi\in L^2(\G_\Q\bs\G_\A)$.
Let $L_{00}^2(\Ga\bs G)$ denote the orthogonal space of all automorphic characters.
\end{definition}

We formulate another condition
\begin{itemize}
\item[(b')] With $G=\G_\A$ and $K=K_\A$ we have
$$
L_0^2(\Ga\bs G)^K\subset L_{00}^2(\Ga\bs G).
$$
\end{itemize}

\begin{lemma}\label{lem3.2}
If $\G$ satisfies (a) and (b), then it satisfies (b').
\end{lemma}

\begin{proof}
Suppose we know the claim in the case $\tilde\G=\G$.
If  $\tilde\G_\A\to\G_\A$ is onto, then every automorphic character of $\G_\A$ induces one of $\G_\A$ and so one sees that the space $L_0^2(\Ga\bs G)$ equals the image of $L_0^2(\tilde\Ga\bs \tilde G)$ and the same holds for $L_{00}^2$ so if $L_0^2(\tilde\Ga\bs \tilde G)^K\subset L_{00}^2(\tilde\Ga\bs \tilde G)$, the same follows for $G$.
So assume that $\G$ is simply connected.
By strong approximation \cite{Rapin} one then has $\G_\A=\G_\Q\G_\R K$. 
So if $\chi$ is a $K$-invariant automorphic character, then $\chi(\G_\A)=\chi(\G_\R)$.
By Lemma 4.7 of \cite{Gorodnik} we have $\chi(\G_\R^0)=1$ and $K\cap \G_\R$ meets every connected component of $\G_\R$, hence $\chi$ is trivial in this case.
\end{proof}

\begin{lemma}\label{lem4.3}
Let $\G$ satisfy the conditions (a) and (b'). Let $G=\G_\A$ as well as $\Ga=\G_\Q$ and $K=K_\A$.
Then the group $G$ acts $K$-mixingly on $\Ga\bs G$.
\end{lemma}

\begin{proof}
Let $L_{00}^2(\Ga\bs G)$ denote the orthogonal space of all automorphic characters and let $\psi$ be a $K$-invariant matrix coefficient.
By condition (b') we can write $\psi=\psi_{v,w}$ for some $v,w\in L_{00}^2$ and then the claim follows from \cite{Gorodnik}*{Theorem 3.26}.
\end{proof}

\begin{definition}
Let $\G$ be a semisimple linear algebraic group.
For a prime number $p$ and $k\in\N_0$ let $D(p^k)$ denote the number of vertices $v$ of $\CB_p$ in the $G_p$-orbit of $K_p$ such that the distance $d_p(v,K_p)$ is $k$.
For a natural number $m$ with prime decomposition $m=p_1^{k_1}\cdots p_t^{k_t}$ let 
$$
D(m)=D(p_1^{k_1})\cdots D(p_t^{k_t}).
$$
\end{definition}

\begin{lemma}
Let $\G$ be a semisimple linear algebraic group over $\Q$ of absolute rank $r$.
Then the Dirichlet series
$$
L(s)=\sum_{m=1}^\infty\frac{D(m)}{m^s}
$$
has a finite abscissa of convergence $B_0<\infty$.
\end{lemma}

\begin{proof}
Embedding the group into some $\SL_n$ and using the corresponding bounds to the Bruhat-Tits building of $\SL_n$ (Lemma \ref{lem4.2}) we deduce that there exists 
a number $n\in\N$ and $C>0$ such that 
$$
D(p)\le C p^n.
$$
As every vertex in the orbit of $K_p$ has $D(p)$ orbits, we deduce
$$
D(p^k)\le C^k D(p)^k
$$
and hence
$$
D(m)\le C^{k_1+\dots+k_t}m^n,
$$
if $m$ has prime decomposition $m=p_1^{k_1}\cdots p_t^{k_t}$.
But then $k_1+\dots+k_t\le \frac{\log m}{\log 2}$, so that
$$
D(m)\le m^{n+\frac{\log C}{\log 2}}.
$$
The lemma follows.
\end{proof}

\begin{theorem}
Let $\G$ satisfy the conditions (a) and (b) of Section \ref{Sec2}.
Let $r$ denote the absolute rank of $\G$ and $r_\R$ its rank over $\R$.
Let $B_0>0$ as in the last lemma and suppose that $B>B_0$.
Then the counting function $N(T)$ of the group $\G$ satisfies
$$
N(T)\ \sim\ \frac{aC}{\vol\(\G_\Q\bs\G_\A\)}(BT)^{r_\R-1} e^{BT},
$$
where $a$ is the area of an explicitly given $r_\R-1$-dimensional simplex and $C=L(B)=\sum_{m=1}^\infty\frac{D(m)}{m^B}$.
\end{theorem}

\begin{proof}
Let $A$ be a maximal connected split torus in $G_\infty=\G(\R)$.
Setting $b^\infty(R)=\vol(B_R^\infty)$ we use the Cartan integral formula \cite{Knapp}*{p.142} to see that $b^\infty(R)$ equals
\begin{align*}
\int_{\a_R^+}\prod_{\al>0}\(\frac{e^{\al(X)}-e^{-\al(X)}}2\)^{m_\al}\, dX,
\end{align*}
where $\a_R^+$ is the set of all $X\in\a^+$ that satisfy $\rho(X)\le BR$.

The leading term of $b^\infty(R)$ in the asymptotic as $R\to\infty$ is
$$
\int_{\a_R^+}\prod_{\al>0}\(\frac{e^{\al(X)}}2\)^{m_\al}\, da=2^{-\sum_{\al>0}m_\al}\int_{\a_R^+}e^{2\rho(X)}\,da
$$

Setting $y=\rho(X)$ we find that this equals
$$
\frac1{2^{\sum_{\al>0}m_\al}}
\int_0^{BR}
F(y)e^y\,dy,
$$
where $F(y)$ equals the area of the $(r_\R-1)$-dimensional submanifold of $\a$ given as $\a_R^+\cap\left\{\rho(X)=y\right\}$.
Here we use that $\a$ carries a euclidean structure inducing the Haar measure.
This structure is not unique, but when we determine the covector $\rho$ to be of length 1, the induced measure on the orthocomplement $\{X:\rho(X)=0\}$ is uniquely determined.
In this way, $F(y)$ is well-defined.
As the equations/inequalities describing this manifold are all linear, we infer that  $F(y)=a^{d-2}$ for $\al_d=F(1)$ which is the volume of a certain $(r_\R-1)$-dimensional simplex in $\a$.

It follows that, as $R\to\infty$, 
\begin{align*}
b^\infty(R)\ \sim\ a\(BR\)^{r_\R-1}e^{BR}.
\end{align*}
Now the Theorem follows from Lemma \ref{lemMeasure} and Theorem \ref{prop1.5}.
\end{proof}

\section{The group PGL(d)}\label{PGLn}
We want to make sure that the conditions of Section \ref{Sec2} are satisfied for the group $\PGL_d$ with $d\ge 2$.
The conditions (a) is clearly satisfied. Condition (b) follows from the next lemma.

\begin{lemma}
Let $R$ be a ring and $2\le d\in\N$. Then 
$\PGL_d(R)=\GL_d(R)/R^\times$ holds for $R$ being a field or $R=\Z_p$ as well as $R=\A_S$ for any set of places $S$.
\end{lemma}

\begin{proof}
The sequence of algebraic groups
$$
1\to\GL_1\to\GL_d\to\PGL_d\to 1
$$
is exact.
By Hilbert's Theorem 90, for each field $k$ the Galois cohomology $H^1(k,\GL_1)$ is trivial and so the exact sequence of Galois cohomology shows that the sequence
$$
1\to\GL_1(k)\to\GL_d(k)\to\PGL_d(k)\to 1
$$
is exact, which implies the claim for fields.
To verify the claim for $R=\Z_p$, we have to analyze the coordinate ring of $\PGL_d$.
First, the coordinate ring of $\GL_d$ over $R$ is
$$
\CO_{\GL_d}= R[x_{ij},y]/\det(x)y-1.
$$
The coordinate ring of $PGL_d=\GL_d/\GL_1$ is the ring of $\GL_1$-invariants, where the action of $\GL_1$ is given by
$$
\la.f(x_{ij},y)=f(\la x_{ij},\la^{-d} y).
$$
The ring of invariants is generated by all monomials of the form $x_{i_1j_1}\cdots x_{i_dj_d}y$ for $1\le i_k,j_k\le d$.
Let now $\chi\in\PGL_d(\Z_p)$.
Then $\chi$ is a homomorphism from $\CO_{\PGL_d}$ to $\Z_p$.
Every such can be extended to $\CO_{\GL_d}\to\Q_p$ and we have to show that there exists an extension mapping to $\Z_p$.
Pick any extension and denote it by the same letter $\chi$.
For the valuation $v$ on $\Q_p$ we have
$$
0\ \le\ v(\chi(x_{ij}^dy))=dv(\chi(x_{ij}))+v(\chi(y)).
$$
We are free to change $\chi(y)$ to $\chi(y)p^{dk}$ for any $k\in\Z$ if at the same time we change $\chi(x_{ij})$ to $\chi(x_{ij})p^{-k}$.
Thus we can assume $v(\chi(y))\in\{ 0,1,\dots,d-1\}$.
Then we conclude $v(\chi(x_{ij}))\ge 0$ for all $i,j$ and so $\chi$ indeed maps into $\Z_p$ as claimed.
This shows the result for $\Z_p$.
The ring $\A_S$ is the restricted product of the fields $\Q_p$, $p\in S$ with respect to the compact subrings $\Z_p$.
So the result for $\Q_p$ and $\Z_p$ implies the result for $\A_S$.
\end{proof}

We want to determine the asymptotic growth of $b(T)$ in the case $\G=\PGL_d$.
By definition we have
$$
b(T)=\sum_{m\le e^T}\left[\prod_{p^k||m}\vol(D_k^p)\right]\vol(B_{T-\log m}^\infty).
$$
Here the sum runs over all $m\in\N$ with $m\le e^T$, the product extends over all prime numbers $p$, where $p^k$ is the exact power of $p$ dividing $m$.
Further $\vol(D_k^p)$ is the volume (=cardinality) of the set
$$
D_k^p=\Big\{ x\in X_p: d_p(x,K_p)=k\Big\}.
$$
Note that $D_k^p=B_k^p\sm B_{k-1}^p$, where $B_k^p$ is the closed ball of radius $k$ in $X_p$.
As $\vol(D_0^p)=1$, the product above is finite.

For $m\in\N$ we write
$$
D(m)=D_d(m)=\prod_{p^k||m}\vol(D_k^p)=\prod_{p^k||m} D(p^k).
$$

\begin{lemma}\label{lem4.2}
For every prime number $p$ and every $k\in\N$ we have
$$
D(p^k)=D(p)c(p)^{k-1}
$$
where  
$$
c(p)=c_d(p)=(d-1)p^{d-1}+p^{d-2}+\dots+p=(d-1)p^{d-1}+\frac{p^{d-1}-1}{p-1}-1.
$$
One has
$$
D(p)=(d-1)(p^{d-1}+\dots+1)=(d-1)\frac{p^d-1}{p-1}.
$$
Let $s_p=s_p(d)=\frac{\log(c(p))}{\log p}$. The series
$$
L(s)=\sum_{m=1}^\infty \frac{D(m)}{m^s}
$$
converges absolutely for $\Re(s)>s_2(d)$ if $d\ge 3$ and for $\Re(s)>2$ if $d=2$.
The function 
$L(s)$ extends to a meromorphic function in $\{\Re(s)>d\}$ with at most simple poles at 
$$
s_p+2\pi i\log(p)k,\quad k\in\Z
$$
where $p$ is a prime number. As the prime $p$ tends to infinity, the sequence  of real parts of the poles, $s_p=\frac{\log(c(p))}{\log p}$, converges monotonically from above to $d-1$, which means that only finitely many of them are $>d$.
The first values for $s_2$ are
$$
\begin{array}{ccc}
n&s_2&s_3\\ \ \\
2&1&1
\\
3&3.3219280949
&2.7712437492
\\
4&4.9068905956&4.1257498573
\\
5&6.2854022189&5.3653166773
\\
6&7.5698556083&6.5507064185
\end{array}
$$
In the special case $d=2$ we have
$$
L(s)=\frac{\zeta(s)\zeta(s-1)}{\zeta(2s)},
$$
where $\zeta$ denotes the Riemann zeta function.
\end{lemma}

\begin{proof}
We show the equivalent formula for the ball volume
$$
\vol(B_k^p)=1+(d-1)\frac{p^d-1}{p-1}\left[\frac{\((d-1)\frac{p^d-1}{p-1}-(d-2)\frac{p^{d-1}-1}{p-1}-1\)^k-1}{(d-1)\frac{p^d-1}{p-1}-(d-2)\frac{p^{d-1}-1}{p-1}-2}\right].
$$
Note that the building $\CB_p$ actually also is the building for the group $\SL_d$.
As it is slightly more convenient here, we will use the latter group for the computations.
We begin by establishing a
formula for the number of chambers of 
$\CB_p$ incident with a fixed $m$-simplex 
$\sigma\subset \CB_p$ where $m \le d-2$. This can 
be found by counting the number of 
minimal parahoric subgroups contained 
in the pointwise stabiliser of $\sigma$ in $G_p=\SL_d(\Q_p)$. 
The pointwise stabiliser of $\sigma$ in 
$G_p$ 
is itself a parahoric subgroup $G_{p,\sigma}$, and 
the number of minimal parahoric 
subgroups that it contains is equal 
to the number of conjugates of a 
fixed minimal parahoric subgroup $G_{p,C}\subset G_{p,\sigma}$, where $C$ is some chamber incident to $\sigma$ and $G_{p,C}$ is its pointwise stabiliser in $G_p$, since the chambers 
incident with $\sigma$ comprise a single $G_{p,\sigma}$-orbit. 
So the desired number may be computed by finding the index in $G_{p,\sigma}$ of the normaliser in $G_{p,\sigma}$ of $G_{p,C}$.
But parahoric subgroups are self-normalising in $G_p$, and so for that reason the number in question is equal to the index of $G_{p,C}$ in $G_{p,\sigma}$. 
The number does not depend on our choice of $C$ and $\sigma$, so it is sufficient to calculate it in the special case where $G_{p,C}$ is the subgroup of $\SL_n(\Z_p)$ such that the entries on or above the main diagonal lie in $\Z_p$ and the entries below the main diagonal lie in the maximal ideal $p\Z_p$, and $G_{p,\sigma}$ is
some strictly larger parahoric subgroup. 
In this case the index of $G_{p,C}$ in 
$G_{p,\sigma}$ may be found by considering the projections to $\SL_d(\F_p)$, and we can calculate it by counting flags of a fixed type in $\P^{d-1}(k)$. 

Thus we obtain the result that the number of chambers of $\CB_p$ incident to a fixed 
$m$-simplex where $m$ is an integer such that $m \le d - 2$, is given by 
$\prod_{j=2}^{d-m}\frac{q^j-1}{q-1}$.
Now let us apply this result to counting the number $b^p_k$ of vertices of distance at most $k$ from a fixed vertex $v$ in the 1-skeleton of $\CB_p$. We clearly have $b_0^p = 1$. 
To calculate $b_1^p$ we must count the number
of edges incident to $v$. 
There are $\prod_{j=2}^d \frac{p^j-1}{p-1}$
chambers incident to $v$ and each one is incident with $d-1$ edges which are incident with $v$. 
But each edge gets counted $\prod_{j=2}^{d-1}\frac{p^j-1}{p-1}$ times, so that we obtain $b_1^p=1+(d-1)\frac{p^d-1}{p-1}$.
To generalise this to the case $k=2$ we must
count how many vertices of distance $2$ from $v$ are joined by an edge
to some fixed vertex $w$ of distance $1$ from $v$. 
There are $(d-1)\frac{p^d-1}{p-1}$ 
vertices joined by an edge to $w$ altogether, 
we must exclude $v$ and vertices at distance $1$ from $v$. 
The number of vertices at distance $1$ from $v$ which are joined by an edge to $w$ is equal to the number of edges incident to $w$ which are incident to chambers incident to $v$ but
not to $v$ itself. 
This number is $(d-2)\frac{p^{d-1}-1}{p-1}$.
We have now justified
the formula in the case $k=2$ and at this point it is clear how to obtain the general result by induction.

The statement on the convergence of the Dirichlet series is obtained as follows.
First,  using  the estimate $k_1+\dots+k_s\le \frac{\log m}{\log2}$, if $m=p_1^{k_1}\cdots p_s^{k_s}$, where the $p_j$ are the different primes dividing $m$, we obtain 
$D(m)\le m^{d-1+\frac{\log d}{\log 2}}$, so that the Dirichlet series converges absolutely for $\Re(s)>>0$.
As the coefficient $D(m)$ is weakly multiplicative, in the domain of absolute convergence the series admits an Euler product of the form
$$
\prod_{p}\(1+\sum_{k=1}^\infty D(p^k)p^{-ks}\).
$$
Write
$$
c(p)=(d-1)p^{d-1}+\frac{p^d-1}{p-1}-1.
$$
For $k\in\N$ we have $D(p^k)=D(p)c(p)^{k-1}$, so that the Euler product becomes
\begin{align*}
\prod_{p}\(1+\sum_{k=1}^\infty D(p)c(p)^{k-1}p^{-ks}\)
&=\prod_{p}\(1+\frac{D(p)}{c(p)}\sum_{k=1}^\infty (c(p)p^{-s})^k\)\\
&=\prod_{p}\(1+\frac{D(p)}{c(p)}\frac{c(p)p^{-s}}{1-c(p)p^{-s}}\)\\
&=\prod_{p}\frac{1-\left[c(p)-D(p)\right]p^{-s}}{1-c(p)p^{-s}}.
\end{align*}
In the case $d=2$ we have $c(p)=p$ and $D(p)=p+1$, so that in this case,
$$
L(s)=\prod_p\frac{1+p^{-s}}{1-p^{-(s-1)}}=
\frac{\zeta(s)\zeta(s-1)}{\zeta(2s)}.
$$
In general, as $c(p),D(p)\le 2(d-1)p^{d-1}$ we infer that the Euler products
$$
\prod_{p}(1-c(p)p^{-s})
$$
and
$$
\prod_{p}\(1-\left[c(p)-D(p)\right]p^{-s}\)
$$
both converge for $\Re(s)>d$ thus extending the holomorphic function $\sum_m\frac{D(m)}{m^s}$ to the range $\Re(s)>d$ minus the set of zeros of the Euler factors $(1-c(p)p^{-s})$.
These zeros are exactly at the points $\frac{\log(c(p))}{\log p}+2\pi i\log(p)k$ and as the function $x\mapsto \frac{\log(c(x))}{\log(x)}$ is monotonically decreasing for $x>1$ we get the extension of $L(s)$.
\end{proof}

In the special case, when the rank of the group $\G$ is one, we can say more.
In this case we also get a result for small values of $B$, which we explain in the examples $\G=\PGL_2$ and $\SL_2$.
 
\begin{theorem}
For the group $\G=\PGL_2$, as $T\to\infty$,
$$
N(T)\ \sim\ \frac{30}{\pi^2}\int_0^\infty e^{-2t}b^\infty(t)\,dt\ \frac1{\vol(\G_\Q\bs\G_\A)}\  e^{2T}.
$$
Equivalently we can write
$$
\pi(x)\ \sim\ \frac{30}{\pi^2}\int_0^\infty e^{-2t}b^\infty(t)\,dt\ \frac1{\vol(\G_\Q\bs\G_\A)}\ x^2.
$$
\end{theorem}

\begin{corollary}
For the group $\G=\SL_2$ we get
$$
\pi(x)\ \sim\ \frac{15}{\pi^2}\int_0^\infty e^{-\frac32t}b^\infty(t)\,dt\ \frac1{\vol(\G_\Q\bs\G_\A)}\ x^{3/2}.
$$
\end{corollary}

\begin{proof}
Let $\G=\PGL_2$.
Then we have
$$
L(s)=\frac{\zeta(s)\zeta(s-1)}{\zeta(2s)}.
$$
This function is holomorphic in $\Re(s)>2$ with a simple pole at $s=2$ of residue $\zeta(2)/\zeta(4)=\frac{15}{\pi^2}$.
A straightforward application of the Wiener-Ikehara Theorem \cite{Chand} yields for any $0\le B<2$ that
$$
\sum_{m\le e^T}\frac{D(m)}{m^B}\ \sim\ \frac{15}{\pi^2} e^{(2-B)T},
$$
as $T\to\infty$.
On the other hand we have $b^\infty(T)\sim e^{BT}$.
Suppose that $0<B<2$. Then we have two Radon measures $\mu$ and $\nu$ on $[0\infty)$ with $\mu([0,T])=\sum_{m\le e^T}\frac{D(M)}{m^B}\sim \frac{15}{\pi^2} e^{(2-B)T}$ and $\nu([0,T])=b^\infty(T)\sim e^{BT}$.

Lemma \ref{lemMeasure} implies that
$$
\lim_{T\to\infty}\frac{b(T)}{e^{2T}}=
\frac{15}{\pi^2}\int_0^\infty e^{-2t}d^\infty(t)\,dt,
$$
which by Theorem \ref{prop1.5} implies the claim of the theorem.
In order to prove the corollary, one
notes that $\SL_2(\Q_p)$ does not act transitively on the vertices of the Bruhat-Tits tree, but has two orbits.
Hence a calculation similar to the $\PGL_2$-case shows that
$$
L(s)=\frac{\zeta(2(s-1))\zeta(2s-1)}{\zeta(4s-2)},
$$
which has a simple pole at $s=\frac 32$ of residue $\frac12$ which gives the claim.
\end{proof}

\begin{corollary}
Let $N(T)$ be the counting function of the group $\G=\PGL_d$ with $d\ge 3$.
If $0<B<d$, we have for any $\delta>0$, as $T\to\infty$,
$$
e^{Ts_2}\ \ll\ N(T)\ \ll\  e^{T(s_2+1+\delta)},
$$
where $s_2=\frac{\log(d\,2^{d-1}-2)}{\log 2}$, see Lemma \ref{lem4.2}.
\end{corollary}

\begin{proof}
We have 
$$
L(s)=\prod_{p}\(1+\sum_{k=1}^\infty D(p)c(p)^{k-1}p^{-ks}\).
$$
Write
$$
\prod_{p\ge 3}\(1+\sum_{k=1}^\infty D(p)c(p)^{k-1}p^{-ks}\)
=\sum_{m=1}^\infty\frac{D^{(2)}(m)}{m^s}.
$$
Then
\begin{align*}
\sum_{m\le e^T}\frac{D(m)}{m^B}
&=\sum_{2^k\le e^T}\frac{D(2)c(2)^{k-1}}{2^{Bk}}\sum_{\substack{m'\le e^T/2^k\\ 2\,\nmid\, m'}}\frac{D^{(2)}(m')}{(m')^B}\\
&\ge \sum_{2^k\le e^T}\frac{D(2)c(2)^{k-1}}{2^{Bk}}\\
&=\frac{D(2)}{c(2)}\sum_{k=0}^{\left[\frac{T}{\log 2}\right]}\(\frac{c(2)}{2^B}\)^k\\
&=\frac{D(2)}{c(2)}
\frac{\(\frac{c(2)}{2^B}\)^{\left[\frac{T}{\log 2}\right]}-1}{\frac{c(2)}{2^B}-1}\\
&\ge\frac{D(2)}{c(2)}
\frac{\(\frac{c(2)}{2^B}\)^{\frac{T}{\log 2}-1}-1}{\frac{c(2)}{2^B}-1}\\
&=\frac{D(2)}{c(2)}
\frac{\frac{2^B}{c(2)}e^{T(s_2-B)}-1}{\frac{c(2)}{2^B}-1},
\end{align*}
so
$$
\sum_{m\le e^T}\frac{D(m)}{m^B}
\gg e^{T(s_2-B)}.
$$
On the other hand, by Lemma \ref{lem4.2} we have that
$$
\sum_{m\le e^T}\frac{D(m)}{m^B}
\ll e^{T(s_2+1+\delta-B)}
$$
for any $\eps>0$.
With the methods of the proof of the last theorem we conclude
$$
e^{Ts_2}\ \ll\ b(T)\ \ll\  e^{T(s_2+1+\delta)},
$$
as $T\to\infty$.
Now the present corollary follows from Corollary \ref{Cor1.5} together with Corollary \ref{Cor2.6}.
\end{proof}

{\small Mathematisches Institut\\
Auf der Morgenstelle 10\\
72076 T\"ubingen\\
Germany\\
\tt deitmar@uni-tuebingen.de,\\ rupert-gordon.mccallum@uni-tuebingen.de}

\today

\end{document}